\documentclass[graybox]{svmult}

\usepackage{mathptmx}       
\usepackage{helvet}         
\usepackage{courier}        
\usepackage{type1cm}        
\usepackage{makeidx}         
\usepackage{graphicx}        
\usepackage{multicol}        
\usepackage{amsmath, amsfonts,  amssymb}

\usepackage{color}

\usepackage{enumerate}

\newcommand{\R}{\mathbb{R}}

\renewcommand{\E}{\mathbb{E}}

\newcommand{\mc}[1]{{\mathcal #1}}

\newcommand{\bb}[1]{{\mathbb #1}}

\newcommand{\<}{\langle}
\renewcommand{\>}{\rangle}

\renewcommand{\epsilon}{\varepsilon}

\renewcommand{\geq}{\geqslant}
\renewcommand{\leq}{\leqslant}

\usepackage[draft]{todonotes}   

\usepackage{comment}
\includecomment{note}

\makeindex             

\begin{document}

\title*{Equilibrium fluctuations for the weakly asymmetric discrete Atlas model}

\author{F. Hern\'andez and M. Simon}
\institute
{
Marielle Simon \at \'Equipe MEPHYSTO, Inria Lille -- Nord Europe, France \textit{and} Laboratoire Paul Painlev\'e, UMR CNRS 8524, Lille, France. \at\email{marielle.simon@inria.fr}
\at
Freddy Hern\'andez \at Instituto de Matem\'atica e Estat\'istica, Universidade Federal Fluminense, Rua M\'ario Santos Braga S/N,
Niter\'oi, RJ 24020-140, Brazil. \at\email{freddyhernandez@id.uff.br}
}

\maketitle

\abstract{This contribution aims at presenting and generalizing a recent work of Hern\'andez, Jara and Valentim \cite{HJV}. We consider the weakly asymmetric version of the so-called \emph{discrete Atlas model}, which has been introduced in \cite{HJV}. Precisely, we look at  some equilibrium fluctuation field of a weakly asymmetric zero-range process which evolves on a discrete half-line, with a source of particles at the origin. 
We prove that its macroscopic evolution is governed by a stochastic heat equation with Neumann or Robin boundary conditions, depending on the range of the parameters of the model.
}

\section{Introduction}
\label{sec:1}

The discrete analogous of the so-called \emph{Atlas model}\footnote{The continuous Atlas model is given by a semi-infinite system of independent Brownian motions on $\mathbb{R}$, see for instance \cite{F,IPB}, and also \cite{HJV} for more details.}
which has been recently introduced in \cite{HJV} is defined as  a family of 
semi-infinite one-dimensional interacting particle systems, and
more specifically, a family of zero-range processes\footnote{We refer to \cite{KL} for a review on zero-range processes.} with a source at the origin.  
Let us first give a formal description of these processes: particles are situated on the semi-infinite lattice $\mathbb{N}=  \{1,2,...,\}$. At each site of the lattice $\mathbb N$ there is a random Poissonian clock of rate 2, which is independent of all the other random clocks attached to the other sites. Each time the clock at site $x \in \mathbb N$ rings, one of the particles at this site moves to one of its two neighbours $x - 1$ or $x + 1$, with equal probability. If the particle decides to move to site $0$, then the particle leaves the system. In addition, with exponential rate $\lambda_n>0$ a particle is created at site $1$. 


We consider in this article the \emph{weakly asymmetric version} of that discrete Atlas model, whose dynamics is now described as follows: whenever the clock at site $x \in \bb N$ rings and there is a particle at this site, one particle jumps to its right neighbouring site $x+1$ with probability $(1-\alpha_n)/2$ and to its left neighbouring site $x-1$ with probability $(1+\alpha_n)/2$, for some $\alpha_n \geqslant 0$ which will be clarified later on. In the same way, at $ x=1$, if the particle tries to jump left, it disappears; besides, a particle is created at site $x=1$ with exponential rate $\lambda_n (1-\alpha_n)$.

 In \cite{HJV} the case $\alpha_n=0$ (i.e.~the symmetric case) and $\lambda_n=1-b/n$ with $b>0$ has been completely investigated. Here we generalize the model by choosing $\alpha_n =a/{n^{\alpha}}$ and $\lambda_n=1-b/{n^\beta}$ for some $a,b\geqslant 0$ and $\alpha,\beta >0$. For this choice of parameters, the equilibrium state of the system is still given by a product geometric measure on $\{0,1,2,\dots\}^{\bb N}$, 	as in \cite{HJV}. 
In this work we show that:  \begin{itemize}
 \item if  $\alpha > 1$ and $\beta  \leqslant 1$, the stationary space-time current fluctuations converge, in a suitable rescaling,  to the solution of the stochastic heat equation with Neuman boundary conditions, similarly to \cite{HJV};
 \item if $\alpha = 1$ and $\beta \leqslant 1$, the stationary space-time current fluctuations converge, in a suitable rescaling,  to the solution of the stochastic heat equation with Robin boundary conditions.
\end{itemize}

Very recently, a mathematical breakthrough has been achieved, towards the \emph{weak KPZ universality conjecture} (see \cite{DieGubPer, gj2014,GJS,GubPer} for instance). This conjecture states that the fluctuations of a large class of
weakly asymmetric one-dimensional interacting particle systems should converge to the KPZ equation. 
The \emph{weakly asymmetric} Atlas model cannot be directly treated by the approach of \cite{gj2014}, because of the presence of a source of particles at one boundary. Therefore, that model needs a special care, which we start here in this work.

The range of the parameters for which we obtain the same macroscopic behaviour of \cite{HJV} may not be sharp. The limit values that we obtain for $\beta$ and $\alpha$, come from the limitation of the main tool that we use, namely the \emph{first order Boltzmann-Gibbs principle}. One could try to improve our results, and treat other ranges of $\beta, \alpha$, by using other estimating tools. This may lead to new macroscopic limits. For instance, one can think about the recent principle stated by \cite{gj2014},  called the \emph{second order Boltzmann-Gibbs principle}, which permits to obtain KPZ-type macroscopic fluctuations, under a stronger asymmetry. We let this prospective research to a future work. Finally, let us mention a related work \cite{fssx} which studies a similar model but with different scalings.

\medskip
 
This article is organized as follows. In Section \ref{sec:framework} we define the zero-range process model with a weak asymmetry and a source at the origin, and we introduce the current fluctuation field in the stationary state. In Section \ref{Sketch}, we sketch the proof and relate the density field to the current field. We also highlight the main differences with the model of \cite{HJV}, namely the presence of possibly diverging boundary terms. In Section \ref{sec:proof} we finally prove the convergence of the density fluctuation field, and we show various estimates related  to the variance of some additive functionals of our dynamics.  The exposition closely follows \cite{HJV}, therefore we omit some details of the proofs that were already included there.

\section{Framework}

\label{sec:framework}

\subsection{Notations}
We denote $\bb N_0:=\{0,1,2,...\}$ and $\bb N:= \{1,2,...,\}$. Our system of particles evolves on $\bb N$, more precisely on each site $ x\in\bb N$ there is a number $\eta_t(x) \in \bb N_0$ of particles which depends on time $t\geq 0$. The state space of the dynamics is therefore $\Omega_0:=\bb N_0^{\bb N}$. 
For any $x,y\in\bb N$ such that $x\neq y$ and $\eta(x) \geq 1$, we define the configuration $\eta^{x,y}$ as being obtained from $\eta \in \Omega_0$ when a particle moves from site $x$ to site $y$, namely: 
\[(\eta^{x,y})(z) = \begin{cases} 
\eta(x)-1 & \text{ if } z=x, \\
\eta(y)+1 & \text{ if } z=y,\\
\eta(z) & \text{ otherwise.}
\end{cases} \]
  We also define $\eta^{0,1}$ and $\eta^{1,0}$ as follows. First, $\eta^{0,1}$ is the configuration obtained from $\eta$ when a particle is created at site 1, namely: 
  \[(\eta^{0,1})(z) = \begin{cases} \eta(1) +1 & \text{ if } z=1, \\ 
  \eta(z) & \text{ otherwise.} \end{cases} \]
  Similarly, if $\eta(1) \geq 1$, the configuration $\eta^{1,0}$ is obtained from $\eta$ when a particle is suppressed at site 1, namely  \[(\eta^{1,0})(z) = \begin{cases} \eta(1) -1 & \text{ if } z=1, \\ 
  \eta(z) & \text{ otherwise.} \end{cases} \]
  We say that a function $\varphi:\Omega_0\to\R$ is \emph{local} if it depends on $\eta \in\Omega_0$ only through a finite number of coordinates. 
  
  \subsection{The microscopic dynamics}
 Let us now rigorously define, as a Markov process $\{\eta_t^n(x) \; ; \; x\in\bb N\}_{t\geq 0}$, the asymmetric dynamics described in the introduction.
  
 Let $g:\bb N_0 \to \bb R$ be given by $g(k)=1$ if $k \neq 0$ and $g(0)=0$, and consider three parameters $\lambda, p,q \in (0,1)$ such that  $p+q=1$. We define an operator acting  on local functions $\varphi:\Omega_0 \to \mathbb R$ as
$$
\mathcal L^{p, q, \lambda} \varphi(\eta) :=  \sum_{x=1}^{\infty} g(\eta(x)) \big[ q \nabla^{x,x+1}\varphi(\eta) + p \nabla^{x,x-1}\varphi(\eta) \big] + \lambda q\; \nabla^{0,1}\varphi(\eta),
$$  
where we denote \[\nabla^{x,y}\varphi(\eta):=\varphi(\eta^{x,y})-\varphi(\eta).\]  
We let the reader refer to \cite{HJV} and check that the Markov process associated with the linear operator $\mc L^{p,q,\lambda}$ is well defined in infinite volume. 
Let $\mu_\lambda$ denote the product geometric measure on $\Omega_0$ given by
\[
 \mu_\lambda(d\eta)=\prod_{x\in\bb N} (1-\lambda) \lambda^{\eta(x)} d\eta(x).
 \]
It is not difficult to see that $\mu_{\lambda}$ is an invariant measure under the evolution of the Markov process generated by $\mc L^{p,q,\lambda}$. 

 Let $n\in\bb N$ be a scaling parameter.  Fix $a, b, \alpha, \beta \in \mathbb R^+$ with $\alpha,\beta >0$, and define
 \begin{equation}
 \label{parameters}
 \lambda_n= 1-\frac{b}{n^\beta}, \ \ \alpha_n=\frac{a}{n^{\alpha}}, \ \ p_n= \frac{1+ \alpha_n}{2},  \ \ q_n= \frac{1-\alpha_n}{2}.
 \end{equation}
These parameters are fixed from now on and up to the end. Finally, let us denote \begin{equation}\label{generator}\mc L_n:= \mc L^{p_n,q_n,\lambda_n},\end{equation} and let $\{\eta_t^n\; ; \; t\geqslant 0\}$ be the \emph{accelerated} process in the time scale 
$tn^\theta$, generated by $n^\theta \mc L_n$ (for some $\theta >0$),  starting from the equilibrium $\mu_{\lambda_n}$. We denote by $\mathbb{P}_n$ the probability distribution of $\{\eta_t^n\; ; \; t\geqslant 0\}$ and by $\mathbb{E}_n$ the corresponding expectation.

%
%

\subsection{Current fluctuation field}
\label{CFF}

For any $x\in\bb N$, let $J_t^n(x)$ be the 
cumulative current of particles between $x$ and $x+1$ up to time $t$, that is, the
signed number of particles which have crossed the bond $\{x,x+1\}$ up to time $t$. Analogously, denote by $J_t^n(0)$ the number of particles created at $x=1$ minus the number of particles that disappeared at $x=1$, up to time $t$. 
 
Our aim is to describe the asymptotic behaviour of the current processes $\{J_t^n(x) \; ; \; t \geqslant 0\}$ as $n$ goes to infinity. For that purpose  we  consider the measure-valued process $\{\mc Z_t^n\; ; \; t\geq 0\}$ defined for any smooth and compactly supported function $f\in\mc C_c^{\infty}([0,\infty),\bb R)$ by
\begin{equation}\label{eq:density}
\mc Z_t^n(f):= \frac{1}{n^{\gamma}} \sum_{x \in \bb N_0} \overline{J_t^n}(x)\; f\big(\tfrac{x}{n}\big),
\end{equation}
where $\gamma>1$ is a parameter that will be made precise below, and $\overline{J_t^n}(x)$ is the recentred current defined as: $\overline{J_t^n}(x)  := J_t^n(x) - \mathbb{E}_n[J_t^n(x)]$.

For a reason that will become clear in Section \ref{Sketch},  instead of  $\mc Z^n_t$ we shall actually work with the field $\{\mc X^n_t \; ; \; t\geqslant 0\}$ defined by
\begin{equation}
 \label{X^n}
 \mathcal X^n_t(f) :=
\frac{1}{n^{\gamma}} \sum_{x \in \bb N_0} \overline{J_t^n}(x) \; f\big(\tfrac{x}{n}\big)
  \;+\;
  \frac{1}{n^{\gamma-1}} \sum_{x \in \mathbb N}\overline{\eta_0^n}(x) F\big(\tfrac{x}{n}\big),
 \end{equation}
  where
$
\overline{\eta_t^n}(x) := \eta_t^n(x) - \mathbb{E}_n[\eta_t^n(x)]
$
and $F(u) :=-\int_{u}^{\infty} f(y) \; dy$.

\subsection{Main results}

Let $\dot{\mc W}_t$ be a standard space-time white noise on $[0,\infty) \times [0,\infty)$. 

\begin{definition}\label{def:mart}Let $A,B\in\bb R$. A measure-valued stochastic process $\{\mc X_t\}_{t \geq 0}$ is said to be a martingale solution of the stochastic heat equation 
\begin{equation} 
\label{eq:heat} 
\partial_t \mc X_t + A \nabla \mc X_t= B \Delta \mc Y_t + \sqrt{2} \dot{\mc W}_t,
\end{equation} 
with boundary condition (BC) at $x=0$ if for any function $f \in \mc C_c^\infty([0,\infty),\bb R)$ satisfying the boundary condition (BC) at $x=0$, the real-valued process 
\[
\mc X_t(f) + A \int_0^t \mc X_s(f')\; ds - B \int_0^t \mc X_s( f'')\; ds
\]
 is a continuous martingale of quadratic variation $2t\int_0^\infty  (f(u))^2 du$.
\end{definition}

Uniqueness of martingale solutions to the stochastic heat equation with Neumann and Robin boundary condition can be found, for instance, in \cite{DT} and \cite[Proposition 2.7]{CorShen}, respectively.

\begin{theorem}
\label{alpha>1}
Let $\alpha >1$ and $\beta \in (0,1]$. 

Assuming $\theta=2+2\beta$  and $\gamma = \beta+\frac32$,
the sequence of processes $\{\mathcal X^n_t; t \geq 0\}_{n \in \mathbb N}$
converges in distribution with respect to the uniform topology to the martingale solution of the stochastic heat equation 
\begin{equation} 
\label{eq:heat1} 
\partial_t \mc X_t = \frac{b^2}{2} \Delta \mc X_t + \sqrt{2} \dot{\mc W}_t,
\end{equation} 
with Neumann boundary condition $f'(0)=0$.
\end{theorem}

\begin{theorem}
\label{alpha=1}
Let $\alpha =1$ and $ \beta \in (0,1]$. 

 Assuming $\theta=2+2\beta$  and $\gamma = \beta+\frac32$,
the sequence of processes $\{\mathcal X^n_t; t \geq 0\}_{n \in \mathbb N}$ defined in \eqref{X^n}
converges in distribution with respect to the uniform topology to the martingale solution of the stochastic heat equation 
\begin{equation} 
\label{eq:heat2} 
\partial_t \mc X_t + ab^2 \nabla \mc X_t= \frac{b^2}{2} \Delta \mc X_t + \sqrt{2} \dot{\mc W}_t,
\end{equation} 
with Robin boundary condition $f(0) - 2f'(0)=0$.
\end{theorem}

Note that  \cite[Theorem 2.1]{HJV} can be recovered from Theorem \ref{alpha>1} by choosing $a=0$ (or $\alpha=\infty$) and $\beta=1$.   

\section{Sketch of the proof}
\label{Sketch}

The proof of Theorem \ref{alpha>1} and Theorem \ref{alpha=1} follows from standard arguments: after proving tightness of the sequence of processes $\{\mc X_t^n\}_n$, one has to show that any of its limit points is a martingale solution of \eqref{eq:heat} with the appropriate initial condition. For that purpose, we write the \emph{martingale decomposition} of $\mc X_t^n$ (see Section \ref{ssec:mart} below), and investigate each of its terms. The  estimate which is going to be crucial is the so-called \emph{first order Boltzmann-Gibbs principle}, which is proved in Section \ref{sec:BG} by using very precise bounds for additive functionals of Markov processes. Since the whole proof is close to \cite{HJV}, we do not copy the exposition of all technical tools which were proved there,  but we refer to them when needed, and we focus on the  estimates which turn out to be different in our case, or even new.

\subsection{Martingale decomposition}\label{ssec:mart}

Using Dynkin's formula, applied to the Markov process $\big\{\big(\eta_t^n, J_t^n(x)\big) \; ; \; t\geqslant 0\big\}$, we can write, for any $x \in \mathbb N$ 
\begin{equation}
\label{Mx}
J_t^n(x) = M^n_t(x) + \int_{0}^t j^n_{x,x+1}(\eta^n_s) \; ds,
\end{equation}
where $j_{x,x+1}(\eta)$ is the microscopic current between sites $x$ and $x+1$ given by \begin{equation} \label{eq:micrcurr} j_{x,x+1}(\eta)=n^{\theta}\Big[q_ng(\eta(x)) - p_ng(\eta(x+1))\Big],\end{equation} and the processes $\{M^n_t(x)\; ; \; t\geqslant 0\}$ are martingales with 
quadratic variation given by
\begin{equation}
\label{qvMx}
\langle M^n_\cdot(x) \rangle_t = n^{\theta} \int_{0}^t \big[q_ng(\eta^n_s(x)) + p_ng(\eta^n_s(x+1))\big] \; ds.
\end{equation}
Let us introduce the centred microscopic current 
\begin{equation} \label{eq:centercurr} \overline{j_{x,x+1}^n} := j^n_{x,x+1} - \mathbb{E}_n[j^n_{x,x+1}] = j^n_{x,x+1} + n^{\theta} \alpha_n \lambda_n, \end{equation} and note that one can easily compute: 
\[\overline{J_t^n}(x)  = J_t^n(x) - \mathbb{E}_n[J_t^n(x)]=  J_t^n(x) +n^{\theta} \alpha_n \lambda_n t. 
\]
For the sake of clarity, in the following we adopt the convention $g(\eta(0))=\lambda_n$. In particular, expressions \eqref{Mx}-\eqref{qvMx} also hold for $x = 0$. Since the currents $J_t^n (x)$ have disjoint jumps, the martingales $\{M^n_t (x) \; ; \; t\geqslant 0\}_{x \in \mathbb N_0}$ are mutually orthogonal. Therefore, with our definition \eqref{eq:density}, for any $f \in \mc C_c^\infty([0,\infty),\bb R)$, the field $\mc Z_t^n(f)$ can be written as
 $$
 \mc Z_t^n(f)= M^n_t(f) + \int_{0}^t \frac{1}{n^{\gamma}} \sum_{x \in \bb N_0} \overline{j^n_{x,x+1}}(\eta^n_s) \; f\big(\tfrac{x}{n}\big) \; ds, 
 $$
where $M^n_t(f)$ is a martingale, given by 
\[ M^n_t(f):= \frac{1}{n^\gamma} \sum_{x\in\bb N_0} M_t^n(x) f\big(\tfrac x n\big), \] and
 with quadratic variation given by
\begin{equation}
\label{qvMf}
\big\langle M^n_\cdot(f) \big\rangle_t = n^{\theta - 2\gamma} \int_{0}^t
\sum_{x \in \bb N_0}   \big[q_ng(\eta^n_s(x)) + p_ng(\eta^n_s(x+1))\big] \; f^2\big(\tfrac{x}{n}\big)\; ds.
\end{equation}
Explicit computations using \eqref{eq:micrcurr}, \eqref{eq:centercurr}, and integration by parts, show that $\mc Z_t^n(f)$ can be decomposed as follows:
\begin{equation} \label{eq:decomp}
 \mc Z_t^n(f) =  \; M^n_t(f) + \mc B_t^n(f) + \mc C_t^n(f) 
 \end{equation}
 where
 \begin{align}
  \mc B_t^n(f) & : = \frac{n^{\theta - \gamma -1} (1+\alpha_n)}{2} \int_{0}^t  \sum_{x \in \bb N} 
 \big( g(\eta^n_s(x)) - \lambda_n\big) \; \nabla^n_{x-1} f \; ds \label{eq:B}\\
\mc C_t^n(f)  & := - \alpha_n n^{\theta - \gamma} \int_{0}^t  \sum_{x \in \bb N} 
 \big( g(\eta^n_s(x)) - \lambda_n\big) \;  f\big(\tfrac{x}{n}\big) \; ds,\label{eq:C}
 \end{align}
where we have used the standard notation for the discrete gradient of $f$: for any $x \in \bb N_0$, $\nabla^n_{x}f := n[f\big(\tfrac{x+1}{n}\big) - f\big(\tfrac{x}{n}\big)]$.

\medskip 
 
 The \emph{first order Boltzmann-Gibbs principle} stated ahead shall allow us to replace $g(\eta^n_s(x)) - \lambda_n$ in \eqref{eq:B} and \eqref{eq:C} above by $b^2n^{-2\beta}(\eta_s^n(x)-\rho_n)$ at a small price (depending on the values of parameters $\theta,\gamma,\alpha,\beta$), obtaining then
 \begin{align}
\nonumber \mc Z_t^n(f) = & \; M^n_t(f) + \frac{b^2 n^{\theta - \gamma -2\beta -1} (1+\alpha_n)}{2} \int_{0}^t  \sum_{x \in \bb N} 
 \big( \eta_s^n(x)-\rho_n \big) \; \nabla^n_{x-1} f \; ds \\
  \label{approx1}
&- b^2\alpha_n  n^{\theta - \gamma- 2\beta} \int_{0}^t  \sum_{x \in \bb N} 
 \big( \eta_s^n(x)-\rho_n \big) \;  f\big(\tfrac{x}{n}\big) \; ds + {o_n(1),}
 \end{align}
 {where $o_n(1)$ denotes a random sequence $(\varepsilon_n)$ which satisfies $\bb E_n[\varepsilon_n^2] \to 0$ as $n\to\infty$.}
 From the continuity relation $J_t^n(x-1)-J_t^n(x)=\eta_t^n(x)-\eta_0^n(x)$, which is valid for any $x\in \bb N$, we have
\begin{equation}
\label{continuity relation 2}
\sum_{x \in \mathbb N} \overline{\eta_t^n}(x) f\big(\tfrac{x}{n}\big)
\; = \;
 \sum_{x \in \mathbb N} \overline{\eta_0^n}(x) f\big(\tfrac{x}{n}\big)
\;+\;
\frac{1}{n}\sum_{x \in \mathbb N}  \overline{J_t^n}(x) \nabla^n_x f
\;+\;
f\big(\tfrac{1}{n}\big)\overline{J_t^n}(0).
\end{equation}
Therefore, the right hand side of \eqref{approx1} can be rewritten as
 \begin{align}
\label{rhs_approx1} & M^n_t(f) +  \frac{b^2 n^{\theta - \gamma -2\beta -1} (1+\alpha_n)}{2} \\ 
&  \nonumber \qquad \qquad \times \int_{0}^t 
\bigg(
\sum_{x \in \mathbb N} \overline{\eta_0^n}(x)  \nabla^n_{x-1} f
\;+\; 
\frac{1}{n}\sum_{x \in \mathbb N}  \overline{J_s^n}(x) \Delta^n_x f
\;+\;
 \nabla^n_{0} f \overline{J_s^n}(0)
  \bigg) \; ds \\ \nonumber
&- b^2\alpha_n  n^{\theta - \gamma- 2\beta} \int_{0}^t  
\bigg(
\sum_{x \in \mathbb N} \overline{\eta_0^n}(x) f\big(\tfrac{x}{n}\big)
\;+\;
\frac{1}{n}\sum_{x \in \mathbb N}  \overline{J_s^n}(x) \nabla^n_x f
\;+\;
f\big(\tfrac{1}{n}\big)\overline{J_s^n}(0)
  \bigg) \; ds + o_n(1),
 \end{align}
 where $\Delta^n_x f$ is the discrete Laplacian of $f$ defined as $\Delta^n_x f:= n[ \nabla^n_{x+1} f- \nabla^n_{x} f]=n^2[f\big(\tfrac{x+1}{n}\big) - 2 f\big(\tfrac{x}{n}\big)+f\big(\tfrac{x-1}{n}\big)]$.

It is exactly this expression that justifies to consider $\mathcal X_t^n$ (see \eqref{X^n} for the definition) instead of $\mathcal Z_t^n$, as stated in the final of Subsection \ref{CFF}. In fact, after replacing in \eqref{rhs_approx1}
 \begin{equation}
 \label{approx2}
 \nabla^n_x f \ \ \text{and} \ \  \Delta^n_x f \quad \ \text{by} \quad \  f'\big(\tfrac{x}{n}\big) \ \ \text{and} \ \  f''\big(\tfrac{x}{n}\big),
 \end{equation}
 respectively (for which the error will be of order $o_n(1)$), we obtain from \eqref{approx1} that the martingale $M_t^n(f)$ reads as
 \begin{align}
  \label{closed}
M^n_t(f) 
\; = \;
 \mathcal X^n_t(f) & - \mathcal X^n_0(f) +c_n \int_{0}^t \mathcal X^n_s(f') ds - d_n  \int_{0}^t \mathcal X^n_s(f'') ds \\
 \label{border}
&-\;
 \big(
  d_n \nabla^n_{0} f 
 - c_n  f\big(\tfrac{1}{n}\big)
 \big)
n^{1-\gamma}  \int_{0}^t  \overline{J_s^n}(0) \; ds + o_n(1),
\end{align}
where
$$
c_n:=b^2 \alpha_n n^{\theta -2\beta -1} 
\quad \ \text{and} \ \quad 
d_n:=\frac{b^2 }{2}n^{\theta -2\beta -2} (1+\alpha_n).
$$
Thanks to this decomposition, we will be able to close the martingale problem \eqref{closed}--\eqref{border}, as explained in the next section. 

\subsection{Closing the martingale problem}

Let us now obtain conditions  on the parameters $\alpha, \beta, \gamma,\theta$ that permit to rewrite \eqref{closed}-\eqref{border} as an approximate closed equation for the fluctuation field $\mathcal X^n_t$. In what follows we denote $\eta_n \approx 1$ if the real sequence $(\eta_n)$ converges to a constant $c \neq 0$ as $n\to\infty$. We say that a function $g(n):\bb N \to \bb R$ is of order $\mc O(\varepsilon_{n})$ if there exists a constant $C>0$ such that, for any $n$, $|g(n)|\leqslant C \varepsilon_{n}$.

\begin{enumerate}[i)]
\item {\bf Quadratic variation of $M^n_t(f)$:} Recall \eqref{qvMf}. 
In order to make $ \mathbb{E}_n[\langle M^n_\cdot(f) \rangle_t]$ converge to $ 2t \int_0^{\infty} f^2(u)du$ as $n\to\infty$, we need to impose 
 \begin{equation}
 \label{thetavsgamma}
 \theta=2\gamma -1.
 \end{equation}
 
 \item{\bf  Initial field $\mathcal X^n_0(f)$:} One can easy check the following: from the Central Limit Theorem (recall that the variables $\{\eta_0(x)\; ; \; x \in \bb N\}$ are i.i.d.~distributed under $\mu_{\lambda_n}$) and the fact that the variance of $\overline\eta_0(x)$, namely $\mathrm{Var}[\overline{\eta}_0(x)]$, is of order $n^{2\beta}$, we have that: if $\gamma - \beta \geq \frac32$
 then $\mathcal X_0^n(f)$ converges, as $n$ goes to infinity, to a Gaussian random variable (possibly degenerate). Moreover if  $\gamma - \beta < \frac32$ this term explodes.

\medskip

\item {\bf Boltzmann-Gibbs principle: } To apply the Boltzmann-Gibbs principle, that is used twice for $\mc B_t^n(f)$ and $\mc C_t^n(f)$ (see Section \ref{sec:BG}), we will need two conditions:
$$
2(\theta -\gamma -1) < \theta - 1 - \frac{3\beta}{2}
\ \ \ \text{ and }\ \ \ 
2( \theta -\gamma -\alpha) < \theta - 1 - \frac{3\beta}{2}.
$$

\item {\bf Real values $c_n$ and $d_n$ :}  
First, note that the conditions $c_n \approx 1$ or $d_n \approx 1$ are equivalent to:
\begin{equation}
\label{cn_vs_dn}
\begin{split}
& c_n \approx 1 \qquad \Leftrightarrow \qquad \theta -2\beta= 1+\alpha \\
& d_n \approx 1 \qquad \Leftrightarrow \qquad \theta -2\beta =2.
\end{split}
\end{equation}
 In addition, since $n \alpha_n  d_n = \frac{1+\alpha_n}{2}c_n 
$, we have: \begin{itemize}
\item if $\alpha \in (0,1)$, 
\[c_n \approx 1   \Rightarrow d_n \approx 0 \qquad \text{and}\qquad d_n \approx 1   \Rightarrow c_n \approx \infty \; ;\]
\item if $\alpha = 1$, 
\[c_n \approx 1  \Leftrightarrow d_n \approx 1\; ;\]
\item if $\alpha \in (1,\infty)$,
\[c_n \approx 1   \Rightarrow d_n \approx \infty \qquad \text{and} \qquad d_n \approx 1   \Rightarrow c_n \approx 0. \]
\end{itemize}
This means that we can already exclude two cases: first, $\alpha \in (0,1)$ and $d_n \approx 1$; and second, $\alpha \in (1,\infty)$ and $c_n \approx 1$.
Furthermore, we can exclude altogether the case $\alpha \in (0,1)$. In fact,  in view of \eqref{thetavsgamma} and \eqref{cn_vs_dn}, if $\alpha \in (0,1)$ and $c_n \approx 1$ then $\gamma-\beta = 1+\frac{\alpha}{2}$, and hence the initial field would explode, as noticed in ii) above.

\medskip

\item {\bf Border term: }
In Section \ref{ConvergenceofJ}, we will prove the following estimate:  
\begin{equation}
\label{varJ(0)}
 \mathbb{E}_n\Big[\big(\ \overline{J_t^n}(0)\big)^2\Big] \leq Ctn^{\theta-1},
\end{equation}
for some positive constant $C>0$. 
Then, observe that \eqref{varJ(0)} combined with \eqref{thetavsgamma} implies that 
 $
 \mathbb{E}_n\big[\big(n^{1-\gamma} \ \overline{J_t^n}(0)\big)^2\big] 
 $
  is bounded. 
  
Furthermore, one can easily see that
$
 \big(
  d_n \nabla^n_{0} f 
 - c_n  f\big(\tfrac{1}{n}\big)
 \big)
$
is of order
\begin{equation}\label{eq:border}
\begin{cases}
n^{\theta - 2\beta -1-\alpha } & \ \text{if} \ \big\{\alpha \in (0,1) \ \text{and} \ f(0) \neq 0 \big\}
\  \text{or} \ \big\{1<\alpha < 2  \ \text{and} \ f'(0) =  0\big\},\\
 n^{\theta - 2\beta -2 }  & \ \text{if} \ \big\{\alpha \in (0,1) \ \ \text{and} \ \ f(0)=0 \big\} \\
 &
 \ \text{or} \ \ \big\{ \alpha=1 \ \ \text{and} \ \ (f'(0)-2af(0))\neq 0\big\} \\
 \phantom{ n^{\theta - 2\beta -3 }} & \ \text{or} \ \ \big\{ \alpha > 1 \ \ \text{and} \ \ f'(0)\neq 0\big\},\\
   n^{\theta - 2\beta -3 } &\ \text{if} \ \big\{\alpha=1 \ \ \text{and} \ \ (f'(0)-2af(0))= 0\big\} 
 \\&  \ \text{or} \ \ \big\{\alpha\geqslant 2 \ \ \text{and} \ \ f'(0)= 0\big\} .
\end{cases}
\end{equation}
\end{enumerate}

To sum up, if we want simultaneously:
\begin{enumerate}[i)]
\item the quadratic variation on $M^n_t(f) $ to be of order $\mc O(1)$ and  not vanish as $n\to\infty$;
\item the initial field $\mathcal X^n_0$ not explode;
\item the Boltzmann-Gibbs principle to be valid;
\item $c_n$ or $d_n$ to be of order $\mc O(1)$ and neither $c_n$ nor $d_n$ goes to infinity as $n$ goes to infinity;
\end{enumerate}
we have to impose:
 \begin{equation}
 \label{impose1}
  \theta=2+2\beta, \quad \gamma = \beta+\tfrac32, \quad \beta \in \big(0,\tfrac43\big).
   \end{equation}
If in addition we want
\begin{enumerate}
\item[(v)] the border term \eqref{border} to go to zero as $n$ goes to infinity, 
in view of \eqref{eq:border} and \eqref{impose1} we have to impose:
  
 \begin{equation}
 \label{border1}
   \  \Big(\alpha=1  \ \text{and}  \ f'(0)-2af(0)=0\Big)
    \ \ \ \text{or} \ \ \   
    \Big(\alpha> 1  \ \text{and} \ f'(0)=0\Big).
 \end{equation}
 \end{enumerate}
 \begin{remark}
Note that Theorems \ref{alpha>1} and \ref{alpha=1} are stated with a stronger condition on $\beta$, namely $\beta \leqslant 1$ instead of $\beta < \frac43$. This assumption only comes out when proving tightness of the sequence $\mc X_t^n$, as explained in the next section.
 \end{remark}
 From now on, the parameters of our model satisfy \eqref{impose1} and \eqref{border1}. With these choices, and recalling \eqref{closed}--\eqref{border}, we have shown that any possible limit point of $\mc X_t^n$ satisfies Definition \ref{def:mart}, with the suitable boundary conditions. The only missing argument to conclude the proof of Theorems \ref{alpha>1} and \ref{alpha=1} is tightness of the processes, which is given in the next section. We will also expose all technical proofs that are still missing, namely the Boltzmann-Gibbs principle, and the proof of \eqref{varJ(0)}.

\section{Proof details and technical estimates}
\label{sec:proof}

In this section we prove the results which will turn the sketch of the proof of Theorems \ref{alpha>1} and \ref{alpha=1}, exhibited in Section \ref{Sketch},  into a rigorous proof. The arguments closely follow  the ones in \cite{HJV}, therefore  the exposition will be inspired by \cite[Sections 3-4]{HJV}, and we will explain the differences on the go.

\subsection{Additive functional estimates}

Let us begin by stating the Kipnis-Varadhan inequality in our context. Note that the infinitesimal generator $\mathcal L_n$ defined in \eqref{generator} can be decomposed as
$ \mc L_n = \mc S_n + \;\alpha_n\; \mc A_n,$ where 
\begin{align*}
\mc S_n \varphi(\eta)& := \sum_{x\in\bb N} g(\eta(x)) \big[\varphi(\eta^{x,x+1})+\varphi(\eta^{x,x-1})-2\varphi(\eta)\big] + \lambda_n \big[\varphi(\eta^{0,1})-\varphi(\eta)\big] \\
\mc A_n \varphi(\eta) & := \sum_{x\in \bb N} g(\eta(x)) \big[\varphi(\eta^{x,x-1})-\varphi(\eta^{x,x+1})\big] - \lambda_n \big[\varphi(\eta^{0,1})-\varphi(\eta)\big].
\end{align*}
Precisely, the operator $\mc S_n$ is symmetric in $\bb L^2(\mu_{\lambda_n})$ and $\mc A_n$ is antisymmetric. The contribution $\alpha_n \mc A_n$ to the operator $\mc L_n$ corresponds to the \emph{weak asymmetry}.

For $f,h \in {\mathbb L}^2(\mu_{\lambda_n})$ we write $\<f,h\>_{\lambda_n} = \int fh  \; d\mu_{\lambda_n}$. We will often omit the sub-index $\lambda_n$ in the scalar product $\<\cdot,\cdot\>$. Let $f \in {\mathbb L}^2(\mu_{\lambda_n})$ such that its average vanishes, namely $\int f \; d\mu_{\lambda_n}=0$. The $\mathcal H^{-1}$-norm of $f$ is defined as
\begin{equation}
\label{varh1}
\|f\|_{-1}^2 = \sup_{h} \big\{ 2\<f,h\> - \< h,(- \mathcal S_n) h \>\big\},
\end{equation}
where the supremum runs over local functions in ${\mathbb L}^2(\mu_{\lambda_n})$. The importance of the $ \mathcal H^{-1}$-norm is shown by the following inequality:

\begin{proposition}[Kipnis-Varadhan inequality]
\label{KV}
For any $T \geq 0$,
\[
\bb E_n \Big[\sup_{0\leq t \leq T} \Big( \int_0^t f(\eta_s^n) ds \Big)^2 \Big] \leq \frac{18 T}{n^{\theta}} \|f\|_{-1}^2.
\]
\end{proposition}
This inequality, in the form presented here was proved in \cite{ChaLanOll} following the proof of a slightly different inequality proved in \cite{KipVar}. 

The strategy used in \cite{HJV} to estimate the variance of additive functionals of the Markov process $\{\eta_t^n(x) \; ; \; x\in\bb N\}_{t\geq 0}$ consists in using the bound provided by the Kipnis-Varadhan inequality and then to estimate the $\mathcal H^{-1}$-norm by using the \emph{spectral gap inequality}. The authors also obtain, without using the spectral inequality, what they called \emph{integration by parts formula} for the $\mathcal H^{-1}$-norm.

Observe that the symmetric part $\mathcal S_n$  corresponds exactly to the same operator studied in \cite{HJV}. In particular, the $\mathcal H_{-1}$-norms considered here and in \cite{HJV} coincide. Consequently, the next two essential technical results obtained in \cite{HJV} still hold in our case, without any modification of the proof: 

\begin{proposition}[{\cite[Proposition 3.4]{HJV}}]\label{prop:bound}
Let $m \in \mathbb{N}$ and $0\leqslant x_0 < \cdots < x_m$ a sequence of sites in $\mathbb{N}_0$. Assume that $\{\varphi_i\}_{i=1,...,m}$ is a sequence of local functions such that the support of $\varphi_i$ is included in $\{x_{i-1}+1,...,x_i\}$ for any $i$. Define $\ell_i=x_i-x_{i-1}$ and assume that $\mathbb{E}[\varphi_i | \eta^{\ell_i}(x)]=0$ for any $x \in \{x_{i-1}+1,...,x_i\}$. Then, 
\[ \mathbb{E}_n\bigg[\sup_{t\in[0,T]} \bigg(\int_0^t \sum_{i=1}^m \varphi_i(\eta_s^n)ds\bigg)^2\bigg] \leqslant \frac{18 \kappa_0 T}{n^\theta}\sum_{i=1}^m \ell_i^2 \mathrm{Var}\big((1+\eta^{\ell_i}(x_{i-1}))\varphi_i\big),\]
where $\kappa_0 >0$ is a constant which is related to the spectral gap of our dynamics.
\end{proposition}

\begin{proposition}[{\cite[Proposition 3.6]{HJV}}] \label{intpart}
Assume that $\sum_{x\in\mathbb{N}_0} \psi^2(x) < +\infty$. Then 
\[   
\mathbb{E}_n\bigg[\sup_{t\in[0,T]} \bigg(\int_0^t \sum_{x\in\mathbb N_0} \big(g_s^n(x)-g_s^n(x+1)\big)\psi(x) \; ds\bigg)^2\bigg] \leqslant \frac{18 T}{n^\theta} \sum_{x\in\mathbb{N}_0} \psi^2(x).
\]
\end{proposition}

Provided with these tools, we first prove tightness (Section \ref{Tightness}), then we state and prove the first order Boltzmann-Gibbs principle (Section \ref{sec:BG}), and finally we control the boundary term (Section \ref{ConvergenceofJ}).

\subsection{Tightness}
\label{Tightness}

Tightness of the sequence $\{\mc X_t^n(F)\; ; \; t\geqslant 0\}$ will follow from Propositions \ref{KV},  \ref{prop:bound} and \ref{intpart} using the same arguments as in \cite[Section 4.2]{HJV}, which will be adapted to our case. We will see that the condition $\beta \leqslant 1$ has to be satisfied in order to get the correct estimates that give tightness. This is exactly at this step that the sharpest condition on $\beta$ appears. 


As in \cite{HJV}, we prove tightness of the sequence  $\{\mc X_t^n\; ;\; t \geq 0\}_{n \in \bb N}$ by restricting ourselves to a finite time horizon $[0,T]$. First, we reduce tightness considerations to real-valued processes, as follows:

\begin{proposition}
\label{realtight}
The family $\{\mc X_t^n\; ;\; t \in [0,T]\}_{n \in \bb N}$ is tight with respect to the uniform topology if and only if for each function $f \in \mc C_c^\infty([0,\infty), \bb R)$ the family of real-valued processes $\{\mc X_t^n(f)\; ;\; t \in [0,T]\}_{n \in \bb N}$ is tight.
\end{proposition}
 A proof of this result can be found in \cite[Chapter 4]{KL}, after an easy adaptation to the infinite volume case.

\medskip

Recall the martingale decomposition \eqref{eq:decomp}, and recall that by definition $\mc X_t^n(f)=\mc Z_t^n(f)+\mc X_0^n(f)$. The proof of tightness for $\mc X_t^n(f)$ can then be reduced to the proof of tightness of the martingale $M_t^n(f)$, the initial distribution $\mc X_0^n(f)$ and the two integral terms $\mc B_t^n(f)$ and $\mc C_t^n(f)$. We treat them as follows: 
\begin{itemize}
\item A simple computation using the fact that $\gamma=\beta+\frac32$ shows that the initial distribution $\mc X_0^n(f)$ converges in distribution to a Gaussian random variable of mean $0$ and variance $\frac{1}{b^2} \int f(u)^2 du$. This integral is finite since $f \in \mc C_c^\infty([0,\infty),\bb R)$.
\item For martingales, powerful methods are available, and $M_t^n(f)$ will be treated thanks to Proposition \ref{martwhitt} below. 
\item The integral terms will be more demanding, but the Kipnis-Varadhan inequality from Proposition \ref{KV} coupled with the integration by parts stated in Proposition \ref{intpart} will provide the necessary bounds, as explained at the end of this paragraph.
\end{itemize} 

Let us start with the convergence criterion for martingales, taken from \cite[Theorem 2.1]{Whi}.

\begin{proposition}
\label{martwhitt}
Let $\{M_t^n\; ;\; t \in [0,T]\}_{n \in \bb N}$ be a sequence of martingales such that $M_0^n \equiv 0$ and let $\Delta_T^n$ the size of the biggest jump of $M_t^n$ in the interval $[0,T]$. Assume:
\begin{itemize}
\item[i)] $\<M_t^n\>$ converges in law to $\sigma^2 t$,
\item[ii)] $\Delta_T^n$ converges in probability to $0$. 
\end{itemize}
Then $\{M_t^n \; ;\; t \in [0,T]\}_{n \in \bb N}$ converges in law to a Brownian motion of variance $\sigma^2$.
\end{proposition}

Moreover, the  Kolmogorov-Centsov's tightness criterion (see for instance \cite[Exercise 2.4.11]{KarShr}) will be useful for the integral term:

\begin{proposition}[Kolmogorov-Centsov's  criterion]
\label{KC}
Let $\{Y_t^n\; ;\; t \in [0,T]\}_{n \in \bb N}$ be a sequence of real-valued processes with continuous paths. Assume that there exist constants $K,a,a' >0$ such that
\[
\bb E\big[|Y_t^n-Y_s^n|^a \big] \leq K |t-s|^{1+a'}
\]
for any $s,t \in [0,T]$ and any $n \in \bb N$. Then the sequence $\{Y_t^n\; ;\; t \in [0,T]\}_{n \in \bb N}$ is tight with respect to the uniform topology.
\end{proposition}

Now we are in position to prove the tightness of $\{\mc X_t^n\; ;\; t \geq 0\}$. Note that the current processes $J_t^n(x)$ have jumps of size $1$. Therefore the jumps of $M_t^n(f)$ are at most of size $\|f\|_\infty /n^{\gamma-1}$. In particular the martingales $M_t^n(f)$ satisfy part $ii)$ of the convergence criterion. Recall  the formula \eqref{qvMf} for the quadratic variation of $M_t^n(f)$. In order to prove $i)$, it is enough to recall that from our assumption: $ \bb E_n \big[\<M_\cdot^n(f)\>_t\big] \to 2t \int_{0}^\infty f^2(u) du$,
and to observe that
\[
\bb E_n \Big[\Big( \<M_\cdot^n(f)\>_t-\bb E_n[\<M_\cdot^n(f)\>_t]\Big)^2\Big] \leq \frac{C t^2}{n^{2\gamma -2}} \sum_{x \in \bb N} f^2\big(\tfrac{x}{n}\big)
\]
for some constant $C$ depending only on $(b,\beta)$. Therefore, not only the martingale  sequence $\{M_t^n(f)\; ;\; t \in [0,T]\}_{n \in \bb N}$ is tight but it also converges to a Brownian motion of variance $2\int f(u)^2 du$.

We are now treating the integral terms. Following \cite{HJV}, let us introduce the definitions $g_s^n(x):=g(\eta_s^n(x))$ and
\begin{align*}
g^\ell(x) &:= \frac{g(\eta(x+ 1))+\dots+ g(\eta(x+ \ell))}{\ell} \\
g_s^{n,\ell}(x) &:= \frac{ g_s^n(x+ 1)+\dots + g_s^n(x+ \ell)}{\ell}.
\end{align*}
Let $h: \bb N_0 \to \bb R$ be such that $\sum_{x} h^2(x) <+\infty$. Notice that the norm $\|\cdot\|_{-1}$ satisfies the triangle inequality. Using the latter twice we can easily see that 
\[
\Big\| \sum_{x \in \bb N_0} \big(g(\eta(x)) -g^\ell(x))) h(x) \Big\|_{-1} 
		 \leq \ell \Big( \sum_{x \in \bb N_0} h(x)^2\Big)^{\frac12}.
\]
Let us start with $\mc B_t^n(f)$. Combining this estimate with Proposition \ref{KV} we obtain  the bound
\begin{equation}
\label{est1}
\bb E_n\Big[\Big(n^{\theta-\gamma-1}\int_0^t \sum_{x \in \bb N_0} \big(g_s^n(x)-g_s^{n,\ell}(x)\big) \nabla_x^n f ds \Big)^2 \Big]
		\leq \frac{18t \ell^2}{n^{2\gamma-\theta+2}} \sum_{x \in \bb N_0} (\nabla_x^n f)^2,
\end{equation}
which is of order $\mc O(\frac{t\ell^2}{n^2})$ from the assumption $\theta=2\gamma-1$. By Cauchy-Schwarz inequality,
\[
\bb E_n\Big[ \Big( n^{\theta-\gamma-1} \int_0^t \sum_{x \in \bb N_0} (g_s^{n,\ell}(x) - \lambda_n) \nabla_x^n f ds \Big)^2 \Big]
		\leq \frac{b t^2 }{\ell n^{\beta-2\theta+2\gamma+2}} \sum_{x \in \bb N_0} (\nabla_x^n f)^2,
\]
which is of order $\mc O(\frac{t^2n^\beta}{\ell})$ from the assumption $\theta=2+2\beta$ and $\gamma=\beta+\frac32$. Choosing $\ell = \lceil n t^{\frac13}\rceil$ we have just proved, under the condition $\beta \leq 1$, that there exists a constant $C:=C(f)$ such that
\[
\bb E_n\Big[\Big( n^{\theta-\gamma-1} \int_0^t \sum_{x \in \bb N}\big(g_s^n(x) -\lambda_n\big) \nabla_x^n f ds \Big)^2 \Big]
		\leq Ct^{\frac53}
\]
for any $n \in \bb N$ and any $t \in [0,T]$. Since the increments of this process are stationary, we have obtained that the hypothesis of Proposition \ref{KC} holds for the integral term $\mc B_t^n(f)$ with $a=2$ and $a' = \frac{2}{3}$. As a consequence, the processes
\[
n^{\theta-\gamma-1} \int_0^t \sum_{x \in \bb N}\big(g_s^n(x) -\lambda_n\big) \nabla_x^n f ds 
\]
are tight.  Note that the integral term $\mc C_t^n(f)$ is easier to treat since it is proportional to  
\[
n^{\theta-\gamma-\alpha} \int_0^t \sum_{x \in \bb N}\big(g_s^n(x) -\lambda_n\big)f\big(\tfrac x n\big) ds 
\]
and we assumed $ \alpha \geq 1$. 

Therefore, we conclude that $\{\mc X_t^n(f)\; ;\; t \in [0,T]\}_{n \in \bb N}$ is tight for any $f \in \mc C_c^\infty([0,\infty),\bb R)$ and by Proposition \ref{realtight} the processes $\{\mc X_t^n\; ; \; t \in [0,T]\}_{n \in \bb N}$ are tight.


\subsection{First order Boltzmann-Gibbs principle} \label{sec:BG}

In this section  we give the crucial argument that permits to perform the replacements in the martingale decomposition \eqref{eq:B}--\eqref{eq:C}. Let us denote by $F=\{{\bf F}_n\}$ a sequence of real-valued functions defined on $\bb N$ such that
\[
 C(F):=\sup_{n\in\bb N} \bigg\{\frac{1}{n} \sum_{x\in\bb N}  {\bf F}_n^2(x) \bigg\} < \infty. 
\]
\begin{proposition}[Boltzmann-Gibbs principle] For any $\delta>0$
\label{prop:firstBG}
\begin{equation}\label{eq:firstBG}
 \bb E_n \bigg[\bigg( \int_0^t \sum_{x\in\bb N} \Big\{g_s^n(x)-\lambda_n-\frac{1}{(1+\rho_n)^2}(\eta_s^n(x)-\rho_n)\Big\} {\bf F}_n(x) \: ds \bigg)^2\bigg] = 
 \mathcal O(n^{\delta - \theta + 1 +\tfrac{3\beta}{2}}).
\end{equation}
\end{proposition}


\begin{proof}
Let $\ell \in \mathbb{N}$ and let us decompose the local function that appears in \eqref{eq:firstBG} as 
\begin{align}
g^n(x+1) - \lambda_n -\frac{1}{(1+\rho_n)^2}  & (\eta^n(x+1)-\rho_n)  = g^n(x+1)-g^{n,\ell}(x) \label{eq:decomp1}\\
& + g^{n,\ell}(x) - \psi^{n,\ell}(x) \label{eq:decomp2}\\
& + \psi^{n,\ell}(x)-\lambda_n-\frac{1}{(1+\rho_n)^2}(\eta^{n,\ell}(x)-\rho_n) \label{eq:decomp3}\\
& + \frac{1}{(1+\rho_n)^2}\big(\eta^{n,\ell}(x)-\eta^n(x+1)\big), \label{eq:decomp4}
\end{align}
where $\eta^{n,\ell}(x)$ and $g^{n,\ell}(x)$ are defined as 
\begin{align*} g^{n,\ell}(x)& = \frac{g(\eta^n(x+1))+\cdots+g(\eta^n(x+\ell))}{\ell}, \\
 \eta^{n,\ell}(x)& = \frac{\eta^n(x+1)+\cdots+\eta^n(x+\ell)}{\ell}\end{align*}
and $\psi^{n,\ell}(x)$ is defined as 
\[ \psi^{n,\ell}(x)=\mathbb{E}\Big[g(\eta^n(x+1))\; \Big|\; \frac1\ell \sum_{y=1}^\ell \eta^n(x+y)\Big].\]
%
%
In what follows  we treat each term \eqref{eq:decomp1}--\eqref{eq:decomp4} separately by means of Lemmas \ref{ref:lem1}, \ref{lem:second}, \ref{lem:3} and \ref{lem:4}, and we estimate their contribution for any $\ell \in \bb N$. Then, we choose:   
\[
\ell =\ell_n := n^{\delta},
\]
and the expectation corresponding to each term \eqref{eq:decomp1}--\eqref{eq:decomp4} is of order \[\mathcal O(n^{3\delta +\frac{3\beta}{2} - \theta + 1}).\] Assuming the validity of these lemmas, the proof of Proposition \ref{prop:firstBG} is concluded.
 \qed 
\end{proof}


\begin{lemma}[First estimate, analogue to {\cite[Lemma 4.5]{HJV}}] Let $\delta >0$ and $\ell=n^\delta$. \label{ref:lem1}
\[\bb E_n\bigg[\bigg(\int_0^t \sum_{x\in\bb N} \big(g_s^n(x)-g_s^{n,\ell}(x-1)\big){\bf F}_n(x)\; ds\bigg)^2\bigg] 
= \mathcal O( n^{ 2\delta- \theta +1})
.\]
\end{lemma}
\begin{proof}
From Proposition \ref{intpart} and the triangular inequality we have 
\[
\bb E_n\bigg[\bigg(\int_0^t \sum_{x\in\bb N} \big(g_s^n(x)-g_s^{n,\ell}(x-1)\big){\bf F}_n(x)\; ds\bigg)^2\bigg] 
\leqslant \frac{18t\ell^2}{n^{\theta}} \sum_{x\in\bb N} {\bf F}_n^2(x) \leqslant 18tC(F)\; \frac{\ell^2}{n^{\theta-1}}.
\]
\end{proof}

\begin{lemma}[Second estimate, analogue to {\cite[Lemma 4.6]{HJV}}] \label{lem:second} Let $\delta >0$ and $\ell=n^\delta$.
\[
\bb E_n\bigg[\bigg(\int_0^t \sum_{x\in\bb N_0} \big(g_s^{n,\ell}(x)-\psi_s^{n,\ell}(x)\big) {\bf F}_n(x+1)\; ds\bigg)^2\bigg] 
= \mathcal O\big(n^{\delta+\beta -\theta +1}(n^\delta+n^{\frac{\beta}{2}})\big).
\]
\end{lemma}
\begin{proof}
Using Proposition \ref{prop:bound}, we can bound
\begin{multline}\bb E_n\bigg[\bigg(\int_0^t \sum_{x\in\bb N_0} \big(g_s^{n,\ell}(x)-\psi_s^{n,\ell}(x)\big) {\bf F}_n(x+1)\; ds\bigg)^2\bigg] \\ \leqslant \frac{18 \kappa_0 \;t C(F) \ell^3}{n^{\theta-1}} \mathrm{Var}\Big[\big (1+\eta^\ell(0)\big)\; \big(g^\ell(0)-\psi^\ell(0)\big)\Big].\label{eq:fb}\end{multline}
In order to simplify notation, here and subsequently we shall omit sub-indexes and write $\mathbb E[\cdot ]=\langle \cdot \rangle_{\lambda_n}$ and $\mathrm{Var}[\cdot ]$ the corresponding variance.
Let us denote $ X:=1+\eta^\ell(0)$, and $Y:= g^\ell(0)-\psi^\ell(0)$.
Then we have 
\[\left\{\begin{aligned}
&\bb E[X] = \frac{n^\beta}{b}  \\
&\textrm{Var}[X] \leqslant \frac{Cn^{2\beta}}{\ell} \\
&\bb E\big[(X-\bb E[X])^4\big] \leqslant \frac{Cn^{4\beta}}{\ell^2} 
\end{aligned}\right. \qquad \text{and} \qquad \left\{\begin{aligned} & \bb E[Y]=0 \vphantom{\frac{n^\beta}{b}} \\ &\textrm{Var}[Y] \leqslant \frac{C}{\ell\; n^\beta} \\ & \bb E[Y^4] \leqslant \frac{C\vphantom{n^\beta}}{n^\beta\; \ell^3}.\end{aligned}\right. \]
Therefore
\begin{align} &\mathrm{Var}\Big[\big(1+\eta^\ell(0)\big)\; \big(g^\ell(0)-\psi^\ell(0)\big)\Big]  \label{eq:var4}\\
& \leqslant \big( \bb E\big[(X-\bb E[X])^4\big]\big)^\frac12 \big(\bb E[Y^4]\big)^\frac12  + 2 \bb E[X] \big(\bb E[Y^4]\big)^\frac12 \big(\textrm{Var}[X]\big)^\frac12 + \big(\bb E[X]\big)^2 \textrm{Var}[Y]\notag \\
& \leqslant C\Big\{ \frac{n^{2\beta}}{\ell^{\frac32}\; n^{\frac\beta 2}\; \ell}+\frac{n^\beta\; n^\beta}{n^{\frac\beta 2}\; \ell^\frac32\; \ell^\frac12} + \frac{n^{2\beta}}{n^\beta\; \ell}\Big\}
\leqslant C\Big\{ \frac{n^{\frac{3\beta}{2}}}{\ell^2} + \frac{n^\beta}{\ell}\Big\}. \notag\end{align}
Replacing this last estimate in \eqref{eq:fb}, this proves Lemma \ref{lem:second}. \qed
\end{proof}

\begin{lemma}[Third estimate,  analogue to {\cite[(4.11)]{HJV}}]\label{lem:3} Let $\delta >0$ and $\ell=n^\delta$.
\begin{multline*}
\bb E_n \bigg[\bigg(\int_0^t \sum_{x\in\bb N_0} \Big(\psi_s^{n,\ell}(x)-\lambda_n-\frac{1}{(1+\rho_n)^2}(\eta_s^{n,\ell}(x)-\rho_n)\Big) F_n(s,x+1)\; ds\bigg)^2\bigg] 
\\ =
\mathcal O( n^{\delta - \theta + 1}).
\end{multline*}
\end{lemma}
\begin{proof}
As before, we have
\begin{multline}
\bb E_n \bigg[\bigg(\int_0^t \sum_{x\in\bb N_0} \Big(\psi_s^{n,\ell}(x)-\lambda_n-\frac{1}{(1+\rho_n)^2}(\eta_s^{n,\ell}(x)-\rho_n)\Big) F_n(s,x+1)\; ds\bigg)^2\bigg] \\ \leqslant \frac{18 \kappa_0 \;t C(F) \ell^3}{n^{\theta-1}} \textrm{Var}\Big[\big(1+\eta^\ell(0)\big)\Big(\psi^{n,\ell}(0)-\lambda_n-\frac{1}{(1+\rho_n)^2}\big(\eta^{n,\ell}(0)-\rho_n\big)\Big)\Big]. \label{eq:mult}
\end{multline}
Let us denote 
\[X:=1+\eta^\ell(0), \qquad Y:= \psi^{n,\ell}(0)-\lambda_n-\frac{1}{(1+\rho_n)^2}\big(\eta^{n,\ell}(0)-\rho_n\big).\]
an use the same estimate for the variance of a product as in used in \eqref{eq:var4}. 

We already know how to bound $\bb E[X]$, $\textrm{Var}[X] $ and $\bb E\big[(X-\bb E[X])^4\big]$. We also have $ \bb E[Y]=0$. The rest of the proof consists in bounding $\textrm{Var}[Y]$ and $\bb E[Y^4]$.
The variance $\textrm{Var}[Y]$ has already been treated in \cite{HJV}, as follows: given $a>b$, let us define $a_n:=\frac{n^\beta}{a}-1$. Then, we have
\begin{align*}
\textrm{Var}[Y]& \leqslant  \frac{1}{(1+\rho_n)^4}\bigg\{\frac{a^2}{n^{2\beta}}\bb E\Big[\big(\eta^{n,\ell}(0)-\rho_n\big)^4\Big]+\rho_n^4\; \mu_{\lambda_n}\Big[\eta^{n,\ell}(0)\leqslant a_n\Big]\bigg\} \\
& \quad + \frac{1}{(\ell-1)^2} \bigg\{ \frac{a^2}{n^{2\beta}}+\mu_{\lambda_n} \Big[\eta^{n,\ell}(0)\leqslant a_n\Big]\bigg\} \\
& \leqslant \frac{C}{n^{2\beta}\; \ell^2} + \frac{C}{\ell^2\; n^{2\beta}} + C\Big(1+\frac{1}{\ell^2}\Big) \mu_{\lambda_n}\Big[\eta^{n,\ell}(0)\leqslant a_n\Big].
\end{align*}
Since $\mu_{\lambda_n}[\eta^{n,\ell}(0)\leqslant a_n]$ decays exponentially fast in $\ell$,
meaning that 
\begin{equation*}
\tfrac{1}{\ell} \log \mu_{\lambda_n}[\eta^{n,\ell}(0)\leqslant a_n] \leq -\mc I_{\rho_n}( a_n )
\ \  \text{where} \ \ 
\lim_{n \to \infty} \mc I_{\rho_n}( a_n)=  \tfrac{b}{a} -\log \tfrac{b}{a} -1,
\end{equation*}
and because $\ell = n^{\delta}$, we have
$
\textrm{Var}[Y] = \mathcal O(n^{-2\delta -2\beta}).
$
We can repeat the same argument, and bound
\begin{align*}
\bb E[Y^4] & \leqslant  \frac{1}{(1+\rho_n)^8}\bigg\{\frac{a^4}{n^{4\beta}}\bb E \Big[\big(\eta^{n,\ell}(0)-\rho_n\big)^8\Big]+\rho_n^8\; \mu_{\lambda_n}\Big[\eta^{n,\ell}(0)\leqslant a_n\Big]\bigg\} \\
& \quad + \frac{1}{(\ell-1)^4} \bigg\{ \frac{a^4}{n^{4\beta}}+\mu_{\lambda_n} \Big[\eta^{n,\ell}(0)\leqslant a_n\Big]\bigg\} \\
& \leqslant \frac{C}{n^{4\beta}\; \ell^4} + \frac{C}{\ell^4\; n^{4\beta}} + C\Big(1+\frac{1}{\ell^4}\Big) \mu_{\lambda_n}\Big[\eta^{n,\ell}(0)\leqslant a_n\Big],
\end{align*}
where we have used the standard bound for moments of geometric distribution,
 \[
 \bb E \Big[\big(\eta^{n,\ell}(0)-\rho_n\big)^8\Big] \leqslant \frac{Cn^{8\beta}}{\ell^4}.
\]
Therefore 
$
\bb E[Y^4] = \mathcal O(n^{-4\delta-4\beta}).
$
Finally, we obtain
\begin{align*}
\textrm{Var}\Big[\big(1+\eta^\ell(0)\big)\Big(\psi^{n,\ell}(0)-\lambda_n-&\frac{1}{(1+\rho_n)^2}\big(\eta^{n,\ell}(0)-\rho_n\big)\Big)\Big] 
= \mathcal O(n^{-2\delta})
\end{align*}
Replacing the last estimate into \eqref{eq:mult}, this proves Lemma \ref{lem:3}. \qed
\end{proof}
\begin{lemma}[Fourth estimate] \label{lem:4} Let $\delta >0$ and $\ell=n^\delta$.
\[\bb E_n\bigg[\bigg(\int_0^t\sum_{x\in\bb N} \frac{1}{(1+\rho_n)^2}\big(\eta_s^{n,\ell}(x-1)-\eta_s^n(x)\big) {\bf F}_n(x)\; ds\bigg)^2\bigg]
= \mathcal O( n^{3\delta - \theta +1})
.\]
\end{lemma}
\begin{proof}
Once more, we use the same inequality, and we bound
\begin{align*}
\bb E_n\bigg[&\bigg(\int_0^t\sum_{x\in\bb N} \frac{1}{(1+\rho_n)^2}\big(\eta_s^{n,\ell}(x-1)-\eta_s^n(x)\big) {\bf F}_n(x)\; ds\bigg)^2\bigg]\\ & \leqslant \frac{1}{(1+\rho_n)^4} \; 
18 \kappa_0 t C(F) 
\frac{\ell^3}{n^{\theta-1}} \textrm{Var}\Big[\big(1+\eta^\ell(0)\big)\big(\eta^\ell(0)-\eta(1)\big)\Big]. 
\end{align*} It is not difficult to prove that $\textrm{Var}\big[\big(1+\eta^\ell(0)\big)\big(\eta^\ell(0)-\eta(1)\big) \big]\leqslant Cn^{4\beta}. $ \qed
\end{proof}

\subsection{Convergence of $\overline{J_t^n}(0)$ }
\label{ConvergenceofJ}
Finally, to conclude the proof, it remains to prove \eqref{varJ(0)}.

 \medskip
 
For that purpose, let $\varphi:(0,\infty) \to [0,\infty)$ be a smooth function of  support contained in $(0,1)$, such that $\int \varphi(x) dx =1$. For $\epsilon >0$ define $\varphi_\epsilon(x) = \frac{1}{\epsilon} \varphi(\frac{x}{\epsilon})$ and $h_\epsilon(x) = \int_x^\infty  \varphi_\epsilon(y)dy$. Note that $h'_\varepsilon(x)=-\varphi_\varepsilon(x)$.

Putting $f$ equal to $h_\epsilon$ in \eqref{continuity relation 2} (which is a consequence of the continuity relation),  and noting that $\nabla_x^n h_{\epsilon}$ is approximately equal to $-\varphi_{\epsilon}(x/n)$ and for $n$ large enough $h_{\epsilon}(1/n) =1$, we see that
\[
\frac{1}{n^{\gamma-1}} \overline{J_t^n}(0) - \mathcal X_t^n(\varphi_\epsilon)  = \frac{1}{n^{\gamma - 1}} \sum_{x\in\bb N} \overline{ \eta_t^n}(x) h_\epsilon \big(\tfrac{x}{n}\big) + \mathrm{Err_n},
\]
with the  error term given by
%
$$
\mathrm{Err_n}=\frac{1}{n^{\gamma}} \sum_{x=1}^\infty \overline{ J_t^n}(x) 
 \big( \nabla_x^n h_{\epsilon} + \varphi_{\epsilon} \big(\tfrac{x}{n}\big) \big),
$$
which satisfies $\bb E_n\big[(\mathrm{Err}_n)^2\big] \to 0$ as $n\to \infty$. 
In particular, since $h_\epsilon(x)$ belongs to $[0,1]$ and  vanishes if $x \geq \epsilon$, since $\textrm{Var}[\eta(x)]$ is of order $n^{2\beta}$ and since we assumed $\gamma = \beta+\frac32$, there is a constant $C$ depending only on the parameters of the model and the choice of $\varphi$ such that
\begin{equation}
\label{current_vs_field}
\bb E_n\Big[ \Big(\frac{1}{n^{\gamma -1}} \overline{J_t^n}(0) - \mathcal X_t^n(\varphi_\epsilon)\Big)^2\Big] \leq C \epsilon.
\end{equation}
Tightness of $\{\mathcal X_t^n\}$, proved in Section \ref{Tightness} assuming $\theta=2+2\beta, \gamma = \beta+\tfrac32, \beta \leqslant 1$ and $\alpha\geq 1$,
 implies boundedness of $\bb E_n\big[ (\mathcal X_t^n(\varphi_\epsilon))^2\big]$, which in turns, after \eqref{current_vs_field}, implies the desired estimate for the border term (namely, inequality \eqref{varJ(0)}).
 
\begin{remark} Note that, as in \cite{HJV}, taking $n \to \infty$ and then $\epsilon \to 0$ in \eqref{current_vs_field} it can be concluded, 
for the case $\alpha>1$, that $n^{1 - \gamma} \overline{J_t^n}(0)$ converges,
in the sense of finite-dimensional distributions, to a fractional Brownian motion of Hurst exponent $\frac{1}{4}$. \end{remark}

\begin{acknowledgement} The authors warmly thank the anonymous referees who helped to improving this work.
The authors also thank the programm \emph{R\'eseau Franco-Br\'esilien en Math\'ematiques} (RFBM) for its support given in 2016, which has enabled this collaboration. 
\end{acknowledgement}

\end{document}